\documentclass[oneside,english]{amsart}

\usepackage{amssymb}
\usepackage{amscd}

\numberwithin{equation}{section} 
\numberwithin{figure}{section} 
\theoremstyle{plain}
\newtheorem*{thm*}{Theorem}
\theoremstyle{plain}
\newtheorem{thm}{Theorem}[section]
\theoremstyle{definition}
\newtheorem{defn}[thm]{Definition}
\theoremstyle{plain}

\theoremstyle{plain}

\theoremstyle{plain}

\theoremstyle{remark}
\newtheorem{rem}[thm]{Remark}
\theoremstyle{remark}
\newtheorem*{acknowledgement*}{Acknowledgement}

\begin{document}

\title[Metrizable isotropic SODE and Hilbert's fourth problem]
{Metrizable isotropic second-order differential equations and Hilbert's fourth problem}

\author[Bucataru]{Ioan Bucataru} \address{Ioan Bucataru, Faculty of
  Mathematics, Alexandru Ioan Cuza University \\ Ia\c si, 
  Romania} \urladdr{http://www.math.uaic.ro/\textasciitilde{}bucataru/}

\author[Muzsnay]{Zolt\'an Muzsnay} \address{Zolt\'an Muzsnay, Institute
  of Mathematics, University of Debrecen \\Debrecen, 
  Hungary} \urladdr{http://www.math.klte.hu/\textasciitilde{}muzsnay/}

\date{\today}

\begin{abstract}
It is well known that a system of homogeneous second-order ordinary
differential equations (spray) is necessarily isotropic in order to be
metrizable by a Finsler function of scalar flag curvature. In Theorem \ref{scalar_flag} we show that the
isotropy condition, together with three other conditions on the Jacobi
endomorphism, characterize sprays that are metrizable by Finsler
functions of scalar flag curvature. The proof of Theorem
\ref{scalar_flag} provides an algorithm to construct the 
Finsler function of scalar flag curvature, in the case when a given
spray is metrizable.  One condition of Theorem \ref{scalar_flag}, regarding the
regularity of the sought after Finsler function, can be relaxed.  By
relaxing this condition, we provide examples of sprays that are
metrizable by conic pseudo-Finsler functions as well as degenerate
Finsler functions.

Hilbert's fourth problem asks to determine the Finsler functions with
rectilinear geodesics. A Finsler function that is a solution to Hilbert's
fourth problem is necessarily of constant or scalar flag curvature. Therefore,
we can use the conditions of \cite[Theorem 4.1]{BM13} and Theorem
\ref{scalar_flag} to test when the projective deformations of a flat spray,
which are isotropic, are metrizable by Finsler functions of constant or scalar
flag curvature. We show how to use the algorithms provided by the proofs of
\cite[Theorem 4.1]{BM13} and Theorem \ref{scalar_flag} to construct solutions
to Hilbert's fourth problem.
\end{abstract}



\maketitle

\section{Introduction}

Second-order ordinary differential equations (SODEs) are important
mathematical objects because they have a large variety of applications in
different domains of mathematics, science and engineering, \cite{AIM93}.  A particularly
interesting class of SODE is the one which can be derived from a variational
principle.  The inverse problem of the calculus of variations (IP) consists of
characterizing variational SODEs, which means to determine whether or not a given
SODE can be described as the critical point of a functional. The most
significant contribution to this problem is the famous paper of Douglas
\cite{Douglas41} in which, using Riquier's theory, he classifies variational
differential equations with two degrees of freedom. Generalizing his results
to higher dimensional cases is a hard problem because the Euler-Lagrange
system is an extremely over-determined partial differential system (PDE), so
in general it has no solution. The integrability conditions of the
Euler-Lagrange PDE can be very complex and can change case by case,
\cite{AT92, BM11, GM00, Krupkova97, MFLMR90, STP02}.  Therefore,
it seems to be impossible to obtain a complete classification of variational
SODE in the $n$-dimensional case, unless we restrict the problem to particular
classes of sprays with special curvature properties, \cite{Berwald41, BR04,
Bryant02, BM13, Crampin07, Shen03}. 

A special and very interesting problem, within the IP, is known as the Finsler
metrizability problem. Here the Lagrangian to search for is the energy
function of a Finslerian or a Riemannian metric, \cite{KS85, Muzsnay06,
  SV02}. Of course, in this problem the given system of SODE and the
associated spray must be homogeneous or quadratic. If the corresponding metric
exists, then the integral curves of the given SODE are the geodesic curves of
the corresponding Finslerian or Riemannian metric.  Since the obstructions to
the existence of a metric for a given SODE are essentially related to
curvature properties of the associated canonical nonlinear connection, it
seems to be reasonable to consider SODEs with special curvature properties.
Obvious candidates to investigate are Finsler structures with constant or
scalar flag curvature.  It is therefore natural to formulate the following
problem. Provide the necessary and sufficient conditions that can be used to
decide whether or not a given homogeneous system of second-order ordinary
differential equations represents the Euler-Lagrange equations of a Finsler
function of constant flag curvature or scalar flag curvature, respectively.
In \cite{BM13} we solved the first part of the problem by giving a
characterization of sprays that are metrizable by Finsler functions of
constant flag curvature.  In the present paper we consider the second part of
the problem and solve it completely by giving a coordinate free
characterization of sprays metrizable by Finsler functions of scalar flag
curvature.  Our main result can be found in Section
\ref{sec:scalar_metrizable}, where we provide the necessary and sufficient
conditions, as tensorial equations on the Jacobi endomorphism, which can be
used to decide whether or not a given homogeneous SODE represents the geodesic
equations of a Finsler function of scalar curvature.
It is known that a spray metrizable by a Finsler function of scalar flag
curvature is necessarily isotropic.  In Theorem \ref{scalar_flag} we provide
three other conditions, which together with the isotropy condition, will
characterize the class of sprays that are metrizable by Finsler functions of
scalar flag curvature.  The proof offers, in the case when the test is
affirmative, an algorithm to construct the Finsler function of scalar flag
curvature that metricizes the given spray. In all the examples we provide, we
show how to use the proposed algorithm to construct such Finsler
functions. The importance of characterizing sprays metrizable by Finsler
functions of scalar flag curvature for constructing all systems of ODEs with
vanishing Wilczynski invariants has been discussed recently in \cite{CDT12}.

In Section \ref{sec:Hilbert} we show that our results for
characterizing metrizable sprays lead to a new approach for Hilbert's
fourth problem. This problem asks to
construct and study the geometries in which the straight line segment is the
shortest connection between two points, \cite{Alvarez05}. Alternatively, one
can reformulate the problem as follows: ''given a domain $\Omega \subset
\mathbb{R}^n$, determine all (Finsler) metrics on $\Omega$ whose geodesics are
straight lines'', \cite[p.191]{Shen01}. Yet another reformulation of
the problem requires to determine projectively flat Finsler metrics,
\cite{Crampin11}. Projectively flat Finsler functions have
isotropic geodesic sprays and therefore have constant or scalar flag
curvature. Such Finsler metrics, of constant flag curvature
where studied in \cite{Shen03}. We use the conditions of 
\cite[Theorem 4.1]{BM13} and Theorem \ref{scalar_flag} to study when
the projective deformations of a flat spray are metrizable. Using
these conditions, we show how to construct 
examples which are solutions to Hilbert's fourth problem by Finsler functions of
constant, and respectively scalar flag curvature.

In Section \ref{sec:examples} we give working examples to show how to use Theorem
\ref{scalar_flag} to test whether or not some other sprays are
Finsler metrizable, and in the affirmative case how to construct the
corresponding Finsler function. By relaxing a regularity condition
of Theorem \ref{scalar_flag}, we show that we can also characterize
sprays that are metrizable by conic pseudo-, or degenerate Finsler
functions.

\section{The geometric framework for Finsler metrizability}

In this section we present the geometric setting for addressing the
Finsler metrizability problem, \cite{BM12a, KS85, Muzsnay06, Shen01,
  Szilasi03}. This geometric setting that includes connections and
curvature can be derived directly from a given homogeneous SODE using
the Fr\"olicher-Nijenhuis formalism, \cite[\S 30]{KMS93},
\cite[Chapter 2]{GM00}. 

\subsection{Spray, connections and curvature}

We consider $M$ a smooth, real and $n$-dimensional manifold. In this
work, all geometric structures are assumed to be smooth. We denote by
$C^{\infty}(M)$ the set of smooth functions on $M$, by
$\mathfrak{X}(M)$ the set of vector fields on $M$,
and by $\Lambda^k(M)$ the set of $k$-forms on $M$. 

For the manifold $M$, we consider the tangent bundle $(TM, \pi, M)$ and
$(T_0M=TM\setminus\{0\}, \pi, M)$ the tangent bundle with the zero
section removed. If $(x^i)$ are local coordinates on the base manifold
$M$, the induced coordinates on the total space $TM$ will be denoted
by $(x^i, y^i)$. 

The tangent bundle carries some canonical structures, very useful to
formulate our geometric framework. One structure is the \emph{vertical
  subbundle} $VTM=\{\xi \in TTM, (D\pi)\xi=0\}$, which induces an
integrable, $n$-dimensional distribution $V: u\in TM \to V_u=VTM\cap
T_uTM$. Locally, this distribution that we will refer to as the
\emph{vertical distribution}, is spanned by $\{\partial /\partial
y^i\}$. Two other structures, defined on $TM$, are the tangent
structure, $J$, and the Liouville vector field, $\mathbb{C}$, locally
given by 
\begin{eqnarray*}
  J=\frac{\partial}{\partial y^i}\otimes dx^i, \quad
  \mathbb{C}=y^i\frac{\partial}{\partial y^i}. \end{eqnarray*}

The main object of this work is a system of $n$ homogeneous second-order
ordinary differential equations, whose coefficients do not depend
explicitly on time, 
\begin{eqnarray}
  \frac{d^2x^i}{dt^2} + 2G^i\left(x,
    \frac{dx}{dt}\right)=0. \label{sode} \end{eqnarray}
For functions $G^i(x,y)$ we assume that they are positive
$2$-homogeneous, which means that 
$G^i(x,\lambda y)=\lambda^2 G^i(x,y)$, for all $\lambda>0$.  By
Euler's Theorem the homogeneity condition of the functions $G^i$ is
equivalent to $\mathbb{C}(G^i)=2G^i$.       

The system \eqref{sode} can be identified with a special vector field
$S\in \mathfrak{X}(T_0M)$ that satisfies the conditions
$JS=\mathbb{C}$ and $[\mathbb{C}, S]=S$. Such a vector field is called
a \emph{spray} and it is locally given by 
\begin{eqnarray}
  S=y^i\frac{\partial}{\partial x^i} - 2G^i(x,y)\frac{\partial}{\partial
    y^i}. \label{slocal} \end{eqnarray}
If we reparameterize the second-order system \eqref{sode}, by preserving the
orientation of the parameter, we obtain a new system and hence a new
spray $\tilde{S}=S-2P\mathbb{C}$, \cite{AIM93, Shen01}. The function $P\in
C^{\infty}(T_0M)$ is $1$-homogeneous, which means that it satisfies
$\mathbb{C}(P)=P$. The two sprays $S$ and $\tilde{S}$ are called
\emph{projectively related}, the function $P$ is called a
\emph{projective deformation} of the spray $S$. 

An important geometric structure that can be associated to a spray is
that of \emph{nonlinear connection} (horizontal distribution,
Ehresmann connection). A nonlinear connection is defined by an $n$-dimensional
distribution $H: u\in TM \to H_u\subset T_uTM$ that is supplementary
to the vertical distribution: $T_uTM=H_u\oplus V_u$.
It is well known that a spray $S$ induces a nonlinear connection with the
corresponding horizontal and vertical projectors given by 
\begin{eqnarray*}
  h=\frac{1}{2}\left( \operatorname{Id} - [S, J]\right), \quad
  v=\frac{1}{2}\left( \operatorname{Id} +  [S, J]\right).  
\end{eqnarray*} 
Locally, the above two projectors can be expressed as follows
\begin{eqnarray*}
  h & =& \frac{\delta}{\delta x^i} \otimes dx^i, \quad
  v=\frac{\partial}{\partial y^i} \otimes \delta y^i, \quad \textrm{ where } \\
  \frac{\delta}{\delta x^i} & = & \frac{\partial}{\partial y^i} - N^j_i(x,y)
  \frac{\partial}{\partial ^j}, \quad \delta y^i=dy^i + N^i_j(x,y) dx^j,
  \quad N^i_j(x,y)=\frac{\partial G^i}{\partial
    y^j}(x,y). \end{eqnarray*}
Alternatively, the nonlinear connection induced by a spray $S$ can be
characterized in terms of an \emph{almost complex structure}, 
\begin{eqnarray*}
  \mathbb{F}=h\circ \mathcal{L}_Sh - J= \frac{\delta}{\delta x^i}\otimes
  \delta y^i - \frac{\partial}{\partial y^i} \otimes dx^i. \label{complexstr}
\end{eqnarray*}
It is straightforward to check that $\mathbb{F}\circ J=h$ and $J\circ
\mathbb{F}=v$. 

The horizontal distribution $H$ is, in general, non-integrable. The
obstruction to its integrability is given by the \emph{curvature tensor}  
\begin{eqnarray}
  R=\frac{1}{2}[h,h]=\frac{1}{2}R^i_{jk}\frac{\partial}{\partial
    y^i}\otimes dx^j \wedge dx^k = \frac{1}{2}\left(\frac{\delta
      N^i_j}{\delta x^k} - \frac{\delta
      N^i_k}{\delta x^j} \right)\frac{\partial}{\partial
    y^i}\otimes dx^j \wedge dx^k 
  \label{curvature} \end{eqnarray}
Due to the homogeneity condition of a spray $S$, curvature
information can be obtained also from the \emph{Jacobi endomorphism}
\begin{eqnarray}
  \Phi=v\circ [S,h] = R^i_j \frac{\partial}{\partial y^i}\otimes dx^j =
  \left(2\frac{\partial G^i}{\partial x^j} - S(N^i_j) -
    N^i_kN^k_j\right) \frac{\partial}{\partial y^i}\otimes
  dx^j. \label{localphi} \end{eqnarray} 
The two curvature tensors are related by 
\begin{eqnarray}
  3R=[J, \Phi], \quad \Phi = i_SR. \label{rphi} \end{eqnarray}
As we will see in this work, important geometric information
about the given spray $S$ are encoded in the \emph{Ricci scalar},
$\rho\in  C^{\infty}(T_0M)$,  \cite{BR04}, \cite[Def. 8.1.7]{Shen01}, which is given by
\begin{eqnarray}
  (n-1)\rho = R^i_i=\operatorname{Tr}(\Phi). \label{ricr} \end{eqnarray}
\begin{defn} \label{defn:iso}
  A spray $S$ is said to be \emph{isotropic} if there exists a semi-basic $1$-form $\alpha
  \in \Lambda^1(T_0M)$ such that the Jacobi endomorphism can be
  written as follows
  \begin{eqnarray}
    \Phi = \rho J-\alpha \otimes \mathbb{C}. \label{isophi} \end{eqnarray}
\end{defn}
Due to the homogeneity condition, for isotropic sprays, the Ricci
scalar is given by $\rho=i_S\alpha$. Using formulae \eqref{rphi} and
\eqref{isophi}, it can be shown that the class of isotropic sprays can
be characterized also in terms of the curvature $R$ of the nonlinear connection,
\cite[Prop. 3.4]{BM11},
\begin{eqnarray}
  3R= \left(d_J\rho +\alpha\right)\wedge J -d_J\alpha \otimes
  \mathbb{C}. \label{isor}\end{eqnarray} 
To complete the geometric setting for studying the Finsler
metrizability problem of a spray, we will use also the \emph{Berwald
  connection}. It is a linear connection on $T_0M$, given by 
\begin{eqnarray*}
  D_XY=h[vX, hY] + v[hX, vY] + (\mathbb{F} +J)[hX, JY] + J[vX,
  (\mathbb{F} +J)], \quad \forall X,Y \in
  \mathfrak{X}(T_0M). \end{eqnarray*}
Locally, the Berwald connection is given by 
\begin{eqnarray*}
  D_{\frac{\delta}{\delta x^i}} \frac{\delta}{\delta x^j} =
  \frac{\partial N^k_i}{\partial y^j} \frac{\delta}{\delta x^k}, \quad
  D_{\frac{\delta}{\delta x^i}} \frac{\partial }{\partial y^j} = 
  \frac{\partial N^k_i}{\partial y^j} \frac{\partial}{\partial y^k}.  \\
  D_{\frac{\partial}{\partial y^i}} \frac{\delta}{\delta x^j}
  = 0, \quad
  D_{\frac{\partial}{\partial y^i}} \frac{\partial }{\partial y^j} = 
  0. \end{eqnarray*}
The Berwald connection has two curvature components. One is the Riemann
curvature tensor and it is directly related to the curvature tensor $R$ and the Jacobi endomorphism
$\Phi$. Another one is the Berwald curvature, \cite[\S 7.1, \S
8.1]{Shen01}. 

\subsection{Finsler spaces}

In this section, we briefly recall the notion of Finsler functions, as well as
some generalizations: conic pseudo-Finsler functions and degenerate Finsler
functions. The variational problem for the energy of a Finsler
function determines a spray, which is called the geodesic spray. The
Finsler metrizability problem requires to decide if a given spray
represents the geodesic spray of a Finsler function.  

\begin{defn} 
  \label{Finsler}
  A continuous function $F:TM \to \mathbb{R}$ is called a \emph{Finsler
    function} if it satisfies the following conditions:
  \begin{itemize}
  \item[i)] $F$ is smooth and strictly positive on $T_0M$, $F(x,0)=0$;
  \item[ii)] $F$ is positive homogeneous of order $1$ in the fibre
    coordinates, which means that $F(x,\lambda y)=\lambda F(x,y)$ for
    all $\lambda > 0$ and $(x,y)\in T_0M$;
  \item[iii)] The $2$-form $dd_JF^2$ is a symplectic form on $T_0M$. 
  \end{itemize}
\end{defn}
In this work we will allow for some relaxations of the above conditions,
regarding the domain of the function as well as the regularity condition iii). See
\cite[\S 1.1.2, \S 1.2.1]{AIM93}, \cite{Bryant02}, \cite{JS12} for more
discussions about the regularity conditions and their relaxation for a Finsler
function.

If the function $F$ is defined on some positive conical region $A\subset TM$
and the three conditions of Definition \ref{Finsler} are satisfied on $A\cap
T_0M$, then we call $F$ a \emph{conic pseudo-Finsler metric}. Moreover, if we
replace the regularity condition iii) by a weaker condition,
$\operatorname{rank}(dd_JF^2)\in \{1,..., 2n-1\}$ on $A\cap T_0M$, we call $F$
a \emph{degenerate Finsler metric}, \cite{JS12}.

\begin{defn} 
  \label{fmetr} 
  A spray $S$ is called Finsler metrizable if there exists a Finsler function
  $F$ such that
  \begin{eqnarray}
    i_Sdd_JF^2=-dF^2. \label{isddj}
  \end{eqnarray}
\end{defn}
We will also use the metrizability property in a broader sense by calling a
spray $S$ conic pseudo-, or degenerate Finsler metrizable if there exist a
conic pseudo-, or degenerate Finsler function $F$ such that the
equation \eqref{isddj} is satisfied. If a spray $S$ satisfies the equation \eqref{isddj}, we call it the
\emph{geodesic spray} of the (conic pseudo-, or degenerate) Finsler function
$F$. It is well known that $S$ is the geodesic spray of such function if and
only if satisfies following equation:
\begin{eqnarray}
  d_hF^2=0. 
  \label{gsf2}
\end{eqnarray}
Consider $S$ the geodesic spray of some (conic pseudo-, or degenerate)
Finsler function $F$ and let $\Phi$ be the Jacobi endomorphism.  
\begin{defn} The function $F$ is said to be of \emph{scalar flag
    curvature}  if there exists a function $\kappa \in
  C^{\infty}(T_0M)$ such that 
  \begin{eqnarray}
    \Phi = \kappa \left( F^2 J - Fd_J F\otimes
      \mathbb{C}\right). \label{fscalar} \end{eqnarray}
\end{defn}

Using formulae \eqref{isophi} and \eqref{fscalar} it follows that for a Finsler
function $F$, of scalar flag curvature $\kappa$, its geodesic spray $S$ is
isotropic, with Ricci scalar $\rho=\kappa F^2$ and the semi-basic
$1$-form $\alpha = \kappa F d_J F.$

Conversely, it can be shown that if an isotropic spray $S$ is metrizable by a
Finsler function $F$, then $F$ is necessarily of scalar flag curvature. See
\cite[Lemma 8.3.2]{Shen01}, or the first implication in \cite[Thm. 4.2]{BM13}
for an alternative proof. One can conclude the above considerations as
follows.

\begin{rem} For a Finsler function, its geodesic spray is
isotropic if and only if the Finsler function is of scalar flag
curvature.\end{rem}

\section{Sprays metrizable by Finsler functions of scalar curvature}
\label{sec:scalar_metrizable}

The problem we want to address in this paper is the following: provide the
necessary and sufficient conditions for a sprays $S$ to be metrizable by a
Finsler function of scalar flag curvature. Above discussion restricts the
class of sprays to start with to the class of isotropic sprays. Alternative
formulations of the conditions we use in the next theorem were proposed first
in \cite[Thm 7.2]{GM00}, in the analytic case, to decide when a non-flat
isotropic spray is variational, by discussing the formal integrability of an
associated partial differential operator. However, next theorem, will provide
an algorithm to construct the Finsler function that metricizes a given spray,
in the case that it is variational. Moreover, the differentiability assumption
we use in the next theorem is weaker, all geometric structure we use are
smooth, not necessarily analytic. Next theorem extends the results of Theorem
4.1 in \cite{BM13}, where we characterize sprays metrizable by Finsler
functions of constant flag curvature.

\begin{thm} 
  \label{scalar_flag}
  Consider $S$ a spray of non-vanishing Ricci scalar.  
  The spray $S$ is metrizable by a Finsler function $F$ of non-vanishing
  scalar flag curvature if and only if 
  \begin{itemize}
  \item[i)] $S$ is isotropic;
  \item[ii)] $d_J(\alpha/\rho) =0$;
  \item[iii)] $D_{hX}(\alpha/\rho) =0$, for all $X\in
    \mathfrak{X}(T_0M)$;
  \item[iv)] $d(\alpha/\rho) +  2i_{\mathbb F}\alpha/\rho \wedge
    \alpha/\rho$ is a symplectic form on $T_0M$. 
  \end{itemize}
\end{thm}
\begin{proof}
  We assume that the spray $S$ is metrizable by  a Finsler function $F$ of scalar flag curvature
  $\kappa$ and we will prove that the four conditions i)-iv) are
  necessary. 

  Since the Jacobi endomorphism $\Phi$ is given by
  formula \eqref{fscalar}, as we discussed already, it follows that $S$
  is isotropic, and hence condition i) is satisfied. 

  The semi-basic $1$-form $\alpha$ and the Ricci scalar $\rho$ are given by 
  \begin{eqnarray} \alpha=\kappa F d_JF, \quad \rho=\kappa
    F^2. \label{ark} \end{eqnarray} It follows that 
  $\alpha/\rho=d_JF/F$ and therefore $d_J(\alpha/\rho)=0$, which
  means that condition ii) is satisfied. 

  Since $S$ is the geodesic spray of the Finsler function $F$, it
  follows from first formula \eqref{gsf2} that $d_hF=0$. Therefore,
  $D_{hX}F=(hX)(F)=(d_h F)(X)=0$ and $D_{hX}d_JF=0$. It follows that
  $D_{hX}(\alpha/\rho) =0$ and hence the condition iii) is also
  satisfied. 

  We check now the regularity condition iv). Using $d_hF=0$ and
  $J\circ \mathbb{F}=v$, we obtain 
  \begin{eqnarray*}
    i_{\mathbb{F}}\frac{\alpha}{\rho}=i_{\mathbb{F}}\frac{1}{F}d_JF=\frac{1}{F}d_vF
    = \frac{1}{F}dF. \label{ifar}
  \end{eqnarray*} 
  Therefore, using the regularity of the Finsler function, it follows
  that  
  \begin{eqnarray*} d\left(\frac{\alpha}{\rho}\right)+ 2i_{\mathbb F}\frac{\alpha}{\rho} \wedge
    \frac{\alpha}{\rho} = d\left(\frac{d_J F}{F}\right)  + \frac{2}{F^2}
    dF \wedge d_JF = \frac{1}{2F^2}
    dd_JF^2 \label{symplectic} \end{eqnarray*} is a
  symplectic form on $T_0M$.

  Let us prove now the sufficiency of the four conditions i)-iv).

  Consider $S$ a spray that satisfies all four conditions i)-iv). First
  condition i) says that the spray $S$ is isotropic, which means that
  its Jacobi endomorphism $\Phi$ is given by formula
  \eqref{isophi}. Next three conditions ii)-iv) refer to the semi-basic
  $1$-form $\alpha$ and the Ricci scalar $\rho$, which enter into the expression \eqref{isophi} of
  the Jacobi endomorphism $\Phi$.

  From condition ii) we have that the semi-basic $1$-form
  $\alpha/\rho$ is a $d_J$-closed $1$-form. Since the tangent
  structure $J$ is integrable, it follows that $[J,J]=0$ and hence
  $d_J^2=0$. Therefore, using a Poincar\'e-type Lemma for the
  differential operator $d_J$, it follows that, locally, $\alpha/\rho$ is a $d_J$-exact $1$-form. It follows that
  there exists a function $f$, locally defined on $T_0M$, such that 
  \begin{eqnarray}
    \frac{1}{\rho}\alpha = d_Jf=\frac{\partial f}{\partial y^i} dx^i. \label{radjf} \end{eqnarray}
  Note that this function $f$ is not unique, it is given up to an
  arbitrary basic function $a\in C^{\infty}(M)$. We will prove that using this
  function $f$ and a corresponding basic function $a$, we can construct a
  Finsler function $F=\exp(f-a)$, of scalar flag curvature, which
  metricizes the given spray $S$. 

  Using the commutation rule for $i_S$ and $d_J$, see \cite[Appendix
  A]{GM00}, we have 
  \begin{eqnarray}
    {\mathbb C}(f)= i_Sd_Jf = i_S\frac{\alpha}{\rho}=1. \label{eq3}
  \end{eqnarray}
  Using the condition ii) of the theorem, and the form \eqref{isor} of the curvature tensor $R$, we obtain
  \begin{eqnarray}
    3d_Rf= (d_J\rho+\alpha )\wedge d_Jf -
    \mathbb{C}(f)  d_J\alpha 
    = (d_J\rho +\alpha )\wedge \frac{\alpha}{\rho} - d_J\alpha =
    - \rho
    d_J\left(\frac{\alpha}{\rho}\right)=0. \label{eq5} \end{eqnarray}

  The condition iii) of the theorem can be written locally as follows
  \begin{eqnarray}
    D_{\delta/\delta x^i}\frac{\partial f}{\partial y^i} =
    \frac{\partial}{\partial y^i} \left(\frac{\delta f}{\delta x^i}\right)
    = 0, \label{dpf} \end{eqnarray}
  which means that the components $\omega_i= \delta f/\delta x^i$ are independent of
  the fibre coordinates. In other words 
  \begin{eqnarray} \omega=d_hf=\frac{\delta f}{\delta
      x^i}dx^i, \label{eq9} \end{eqnarray}
  is a basic $1$-form on $T_0M$. Using formula \eqref{eq5} we 
  \begin{eqnarray}
    0=d_Rf=d^2_hf=d_h(d_h f) = \frac{1}{2}\left(\frac{\partial
        \omega_i}{\partial x^j} - \frac{\partial
        \omega_j}{\partial x^i} \right) dx^i\wedge
    dx^j=d(d_hf). \label{eq10} \end{eqnarray}
  It follows that the basic $1$-form $d_hf \in \Lambda^1(M)$ is
  closed and hence it is locally exact. Therefore, there exists a
  function $a$, which is locally defined on $M$, such that
  \begin{eqnarray}
    d_hf=da = d_ha. \label{dhaf} \end{eqnarray}
  We will prove now that the function 
  \begin{eqnarray}
    F=\exp(f-a),\label{expf} \end{eqnarray} locally defined on
  $T_0M$, is a Finsler function of scalar flag curvature, whose geodesic
  spray is the given spray $S$. Depending on the domain of
    the two functions $f$ and $a$, the function $F$ might be a conic
    pseudo-Finsler function.

  From formula \eqref{eq3}, we have that 
  $ \mathbb{C}(F)= \exp(f-a) \mathbb{C}(f)=F$, which means
  that $F$ is $1$-homogeneous. Using formula \eqref{dhaf}, we obtain that 
  \begin{eqnarray}
    d_hF= \exp(f-a) d_h(f-a)=0. \label{dhfinsler} \end{eqnarray}

  The semi-basic $1$-form $\alpha/\rho$, which is given by formula
  \eqref{radjf}, can be expressed in terms of the function $F$, given by
  formula \eqref{expf}, as follows 
  \begin{eqnarray*}
    \frac{\alpha}{\rho}=\frac{d_JF}{F}. \label{far} \end{eqnarray*}
  We use now formula \eqref{dhfinsler} and obtain
  \begin{eqnarray} d\left(\frac{\alpha}{\rho}\right)+ 2i_{\mathbb F}\frac{\alpha}{\rho} \wedge
    \frac{\alpha}{\rho} = \frac{1}{F^2}
    dd_JF^2. \label{symplectic} \end{eqnarray} The last condition of the 
  theorem assures that $dd_JF^2$ is a symplectic form and hence $F$ is a
  Finsler function.
  Due to formula \eqref{dhfinsler}, we obtain that $S$ is the geodesic
  spray of the Finsler function $F$. 

  To complete the proof, we have to show now that $F$ has
  non-vanishing scalar flag
  curvature.  Since the Finsler function $F$ is given by formula
  \eqref{expf}, we have that $F>0$ on $T_0M$ and we may consider the
  function
  \begin{eqnarray}
    \kappa=\frac{\rho}{F^2}. \label{krf} \end{eqnarray}
  It follows that the semi-basic $1$-form $\alpha$ is given by 
  \begin{eqnarray}
    \alpha=\frac{\rho}{F}d_JF = \kappa F d_JF. \label{akf}\end{eqnarray}
  Since the Ricci scalar does not vanish, it follows that the function
  $\kappa$ has the same property. The last two formulae \eqref{krf} and \eqref{akf} show that the Jacobi
  endomorphism $\Phi$, of the geodesic spray $S$ of the Finsler function
  $F$, is given by formula \eqref{fscalar}. Therefore, the Finsler
  function $F$ has non-vanishing scalar flag curvature $\kappa$.
\end{proof}

We can replace the regularity condition $iv)$ of the Theorem
\ref{scalar_flag} by a weaker condition and require that
$\operatorname{rank}\left( d(\alpha/\rho) +  2i_{\mathbb F}\alpha/\rho \wedge
    \alpha/\rho \right) \neq 0$ on some conical region in $T_0M$. In
  this case the theorem provides a characterization for sprays
  metrizable by conic pseudo-, or degenerate Finsler function. We
  consider two examples of such sprays in Section \ref{sec:examples}.

For dimensions grater than two, the Theorem \ref{scalar_flag} does not address the Finsler metrizability
problem in its most general context. The cases that are not covered by
this theorem refer to sprays that are metrizable by Finsler functions
which do not have scalar flag curvature. 

However, in the $2$-dimensional case, the Theorem \ref{scalar_flag}
covers the Finsler metrizability problem in the most general case. This is
due to the fact that any $2$-dimensional spray is isotropic and
therefore, the Finsler metrizability problem is equivalent to
the metrizability by a Finsler function of scalar flag curvature. For the two-dimensional case, in \cite{Berwald41}, Berwald provides
necessary and sufficient conditions, in terms of the curvature scalars, such that the extremals of a
Finsler space are rectilinear.

The importance of characterizing sprays that are metrizable by Finsler
functions of scalar flag curvature was discussed recently in
\cite{CDT12} since it will allow to ''construct all systems of ODEs with vanishing Wilczynski invariants''.

\section{Hilbert's fourth problem}
\label{sec:Hilbert}

''Hilbert's fourth problem asks to construct and study the geometries
in which the straight line segment is the shortest connection between
two points", \cite{Alvarez05}. Alternatively, the problem can be
reformulated as follows: ''given a domain $\Omega \subset \mathbb{R}^n$, determine all (Finsler)
metrics on $\Omega$ whose geodesics are straight lines'',
\cite[p.191]{Shen01}. These Finsler metrics are projectively flat and can
be studied using different techniques, \cite{Crampin11, CMS13, Shen03}. All
such Finsler functions have constant or scalar flag
curvature. Therefore, we can use the conditions of \cite[Thm. 4.1]{BM13}  and Theorem
\ref{scalar_flag} to test when a projectively flat
spray is Finsler metrizable. For such sprays we use the algorithms provided in the
proofs of \cite[Thm. 4.1]{BM13}  and Theorem
\ref{scalar_flag} to construct solutions to Hilbert's fourth problem. 

We start with $S_0$, the flat spray on some domain $\Omega \subset \mathbb{R}^n$. A projective
deformation $S=S_0-2P\mathbb{C}$ leads to a new spray that is
isotropic. In the case that the spray $S$ satisfies the metrizability
tests of either \cite[Thm. 4.1]{BM13} or Theorem
\ref{scalar_flag},  then $S$ is the geodesic spray of a Finsler
function of constant or scalar flag curvature. This way we provide a
method to construct Finsler functions of constant or scalar flag
curvature with rectilinear geodesics. 

Consider a domain $\Omega\subset \mathbb{R}^n$ and let $S_0\in
\mathfrak{X}(\Omega \times \mathbb{R}^n)$ be the flat spray.
We will study now, when a projective deformation 
\begin{eqnarray} S=S_0 -
2P\mathbb{C}=y^i\frac{\partial}{\partial x^i} -
Py^i\frac{\partial}{\partial y^i}, \label{pfspray} \end{eqnarray} for
a $1$-homogeneous function $P \in C^{\infty}(\Omega \times
\mathbb{R}^n\setminus\{0\})$, leads to a metrizable 
spray $S$ by a Finsler function $F$ of constant flag curvature. Such 
Finsler function $F$ will be then a solution to Hilbert's fourth problem.

Using formulae \cite[(4.8)]{BM12a}, the
Jacobi endomorphism of the new spray $S$ is given by
\begin{eqnarray}
  \Phi = (P^2 - S_0P) J - (Pd_JP  + d_J(S_0P) - 3d_{h_0}P)\otimes
  \mathbb{C}. \label{pp0} \end{eqnarray}
It follows that the spray $S$ is isotropic, the Ricci scalar, $\rho$, and
semi-basic form $\alpha$ are given by:
\begin{eqnarray} \rho=P^2 - S_0P, \ \alpha = 
  Pd_JP  + d_J(S_0P) - 3d_{h_0}P. \label{raflat} \end{eqnarray}
From above formula it follows that 
\begin{eqnarray}
  d_J\alpha =-3d_Jd_{h_0}P= 3d_{h_0}d_JP. \label{dja0}\end{eqnarray}
Using formula \cite[(4.8)]{BM12a}, the corresponding horizontal
projectors for the two sprays $S$ and $S_0$ are related by 
\begin{eqnarray}
  h=h_0-PJ - d_JP\otimes \mathbb{C}. \label{hh0} \end{eqnarray}
We use that $\mathbb{C}(P^2-S_0P)=2 (P^2-S_0P)$ as
well as the formulae \eqref{raflat} and \eqref{hh0} to obtain
\begin{eqnarray}
  d_h\rho=d_{h_0}(P^2-S_0P) - Pd_J\rho - 2\rho d_JP. \label{dhr0} 
\end{eqnarray}
In Subsection \ref{subsec:h4cc} we will use the conditions of \cite[Thm. 4.1]{BM13} to test if
the spray $S$, given by formula \eqref{pfspray},  is metrizable by a Finsler function of constant flag
curvature. In Subsection \ref{subsec:h4sc} we will use the conditions of Theorem \ref{scalar_flag} to test if
the spray $S$ is metrizable by a Finsler function of scalar flag
curvature. In each subsection, we show how to construct examples of sprays that are
metrizable by such Finsler functions.

\subsection{Solutions to Hilbert's fourth problem by Finsler functions
  of constant flag curvature} \label{subsec:h4cc}

The projectively flat spray $S$, given by formula \eqref{pfspray}, is
isotropic, the Ricci scalar, $\rho$, and the semi-basic $1$-form
$\alpha$ are given by formulae \eqref{raflat}.
According to \cite[Thm. 4.1]{BM13}, the spray $S$ is metrizable by a Finsler function of
constant flag curvature if and only if the following three conditions are
satisfied:
\begin{itemize}
\item[C1)] $d_J\alpha=0$;
\item[C2)] $d_h\rho=0$;
\item[C3)] $\operatorname{rank}(dd_J\rho)=2n$.
\end{itemize}
We study now the first condition $C1)$. Since the spray $S_0$ is flat, it
follows that $R_0=[h_0,h_0]/2=0$ and therefore $d^2_{h_0}=0$. Using a
Poincar\'e-type Lemma for the differential operator $d_{h_0}$, and formula
\eqref{dja0}, it follows that the condition $C1)$ is satisfied if and only if
there exists a locally defined, $0$-homogeneous, smooth function $g$ on
$\Omega \times \mathbb{R}^n\setminus\{0\}$ such that
\begin{eqnarray}
  d_JP=d_{h_0}g. \label{c1}\end{eqnarray}
From the above formula, by composing with the inner product $i_{S_0}$
to the both sides, we obtain 
\begin{eqnarray}
  P=\mathbb{C}(P)=i_{S_0}d_JP=i_{S_0}d_{h_0}g=S_0(g). \label{psog} \end{eqnarray}
In view of this formula, we obtain that the Ricci scalar, $\rho$, in
formula \eqref{raflat}, can be expressed as follows:
\begin{eqnarray}
  \rho = \left(S_0(g)\right)^2 - S^2_0(g). \label{rhosog} \end{eqnarray} 
Using formula \eqref{dhr0}, as well as the above formulae, we obtain that the second condition $C2)$
is satisfied if and only if 
\begin{eqnarray*}
  d_{h_0}\rho - S_0(g) d_J\rho - 2\rho d_{h_0}g=0. \end{eqnarray*}
We can write above formula, which is equivalent to the condition $C2)$,
as follows
\begin{eqnarray}
  d_{h_0}(\exp(-2g)\rho)+\frac{1}{2}
  S_0(\exp(-2g))d_J\rho=0. \label{c2} \end{eqnarray}
\begin{rem} \label{h4eqcc} Each solution $g$ of the above equation \eqref{c2}
  determines a projectively flat Finsler metric
  $F^2=\left|\left(S_0(g)\right)^2 - S^2_0(g)\right|$, of constant
  flag curvature, if and only if the regularity condition $C3)$ is
  satisfied. \end{rem} Next, we provide some examples of such
functions $g$. 

\subsubsection{Example} \label{ssc:ex1} Consider the open disk $\Omega=\{x\in \mathbb{R}^n, |x|<1\}$, the
function $g(x)=-\ln \sqrt{1-|x|^2}$, and the projectively flat  spray
$S=S_0-2g^c\mathbb{C}\in \mathfrak{X}(\Omega \times
\mathbb{R}^n)$. 
The particular form of the projective factor
$P(x,y)=g^c(x,y)=S_0(g)=y^i \partial g/\partial x^i$ assures that
the the function $g$ is a solution of the equation \eqref{c1}, which
means that the condition $C1)$ is satisfied.

For the spray $S$, the Ricci scalar given by formula
\eqref{rhosog} has the following expression
\begin{eqnarray}
  \rho(x,y)=-\frac{|y|^2(1-|x|^2)+<x,y>^2}{(1-|x|^2)^2}. \label{rhoklein}
\end{eqnarray}
Since the function $g$ is a solution of the equation
\eqref{c2} it follows that the condition $C2)$ is satisfied. 

It remains to check the regularity condition $C3)$. By a
direct computation we have $dd_J\rho = 2g_{ij}\delta y^i \wedge dx^j$,
where 
\begin{eqnarray}
  g_{ij} = \frac{\partial^2 g}{\partial x^i\partial x^j}  - \frac{\partial
    g}{\partial x^i} \frac{\partial g}{\partial x^j} =
  \frac{1}{1-|x|^2}\left(\delta_{ij} +
    \frac{x_ix_j}{1-|x|^2}\right), \end{eqnarray} is the Klein
metric on the unit ball, see \cite[Example
11.3.1]{Shen01}.
Therefore, we have that the projectively flat spray $S$ is the
geodesic spray of the Kein metric, 
\begin{eqnarray} F^2(x,y)=-\rho(x,y)=  \frac{|y|^2(1-|x|^2)+<x,y>^2}{(1-|x|^2)^2}, \label{fklein}
\end{eqnarray}
which has constant flag curvature $\kappa=\rho/F^2=-1$.

\subsubsection{Example} \label{ssc:ex2} If we consider the function
$g(x)=-\ln\sqrt{1+|x|^2}$, solution of equations \eqref{c2}, we obtain that the
spray $S=S_0-2g^c\mathbb{C}\in \mathfrak{X}(\mathbb{R}^n\times
\mathbb{R}^n)$ is metrizable by the following metric on
$\mathbb{R}^n$, 
\begin{eqnarray}
  F^2=S_0(g^c)-(g^c)^2=\frac{|y|^2(1+|x|^2) -
    <x,y>^2}{(1+|x|^2)^2}, \label{fpos} \end{eqnarray}
of constant curvature $\kappa=1$, \cite[Example
11.3.2]{Shen01}.

\subsection{Solutions to Hilbert's fourth problem by Finsler functions
  of scalar flag curvature} \label{subsec:h4sc}

In this subsection, we try to extend the question we
addressed in the previous subsection, from constant flag curvature to
scalar flag curvature. Therefore, we consider a domain $\Omega\subset
\mathbb{R}^n$ and let $S_0\in
\mathfrak{X}(\Omega \times \mathbb{R}^n)$ be the flat spray.
We will provide an example of a projective deformation $S=S_0 -
2P\mathbb{C}$, for a $1$-homogeneous function $P \in C^{\infty}(\Omega
\times \mathbb{R}^n)$, which will lead to a
spray metrizable by Finsler functions of scalar flag curvature. Such 
projectively flat Finsler function, will be therefore a solution to Hilbert's fourth problem. 

As we have seen already, the spray $S=S_0 -
2P\mathbb{C}$ is isotropic, the Ricci scalar, $\rho$, and the semi-basic $1$-form
$\alpha$ are given by  formulae \eqref{raflat}. Since the spray $S$ is isotropic, according to
Theorem \ref{scalar_flag}, it follows that $S$ is Finsler metrizable,
which is equivalent to be metrizable by a Finsler function of scalar
Flag curvature, if and only if the following three conditions are satisfied:
\begin{itemize}
\item[S1)] $d_J(\alpha/\rho)=0$;
\item[S2)] $\mathcal{D}_{hX}(\alpha/\rho)=0$;
\item[S3)] the regularity condition iv) of Theorem \ref{scalar_flag}. 
\end{itemize}
Next we provide an example of a projective factor $P$, which has a
very similar form with those considered in the previous two examples. However, for
the function $P$ in the next example, the projectively flat spray $S$ satisfies the
conditions $S1)$, $S2)$, and $S3)$ and hence will be metrizable by a
Finsler function of scalar flag curvature.  

\subsubsection{Example} \label{ssc:ex3} For the open disk $\Omega=\{x\in
\mathbb{R}^n, |x|<1\}$ in $\mathbb{R}^n$, we consider the function
$g \in C^{\infty}(\Omega \times
(\mathbb{R}^n\setminus\{0\}))$, and the projectively flat  spray
$S\in \mathfrak{X}(\Omega \times
(\mathbb{R}^n\setminus\{0\}))$, given by 
\begin{eqnarray}
  g(x,y)=\ln \sqrt{|y|+<x,y>}, \quad
  S=S_0-2S_0(g)\mathbb{C}=y^i\frac{\partial}{\partial x^i} -
  \frac{|y|^2y^i}{|y|+<x,y>} \frac{\partial}{\partial y^i}. \label{ex3g}
\end{eqnarray} 
The projective factor $P=S_0(g)$ is given by 
\begin{eqnarray*}
  P(x,y)=\frac{1}{2}\frac{|y|^2}{|y|+<x,y>}. \end{eqnarray*}
Using first formula \eqref{raflat} we obtain that the Ricci scalar is
given by 
\begin{eqnarray}
  \rho(x,y)=3P^2(x,y) =
  \frac{3}{2}\frac{|y|^4}{(|y|+<x,y>)^2}. \label{ex3rho} \end{eqnarray}
Using above formula for $\rho$ and the second formula \eqref{raflat} we obtain that the semi-basic
$1$-form $\alpha$ is given by 
\begin{eqnarray}
  \alpha = -3(d_{h_0}P + Pd_JP)= \frac{3|y|^2}{4(|y|+<x,y>)^3}(y_i|y| +
  x_i|y|^2)dx^i. \label{ex3alpha} \end{eqnarray} 
Using formulae \eqref{ex3rho} and \eqref{ex3alpha} it follows that the
semi-basic $1$-form $\alpha/\rho$ is $d_J$-closed, since 
\begin{eqnarray}
  \frac{\alpha}{\rho} = \frac{1}{|y|+<x,y>}\left(\frac{y_i}{|y|} +
    x_i\right)dx^i = d_J f, \quad f(x,y)=\ln(|y|+<x,y>). \label{ex3ar} \end{eqnarray}

From the above formula we have that the first condition $S1)$ is
satisfied. As we have shown in the proof of Theorem \ref{scalar_flag},
the second condition $S2)$ is equivalent to the fact that $d_hf$ is a basic
$1$-form on $\Omega$. Using formula \eqref{hh0} for the horizontal
projector $h$ and expression \eqref{ex3ar} for the function $f$, we
have that $d_hf=0$. The regularity condition $S3)$ is also satisfied
and hence, by formula \eqref{expf}, we obtain that 
\begin{eqnarray}
  F(x,y)=\exp f(x,y)=|y|+<x,y>, \label{ex3finsler} \end{eqnarray} is a
Finsler function. The function $F$ is a Finsler function of Numata
type, see \cite[3.9. B]{BCS00}. The Finsler function $F$ has scalar
flag curvature, which is given by
formula \eqref{krf}, 
\begin{eqnarray}
  \kappa(x,y)=\frac{\rho}{F^2} =
  \frac{3}{4}\frac{|y|^4}{(|y|+<x,y>)^4}. \label{ex3sf}\end{eqnarray}  
The geodesics of the Finsler function $F$, given by formula \eqref{ex3finsler}, are segments of straight
lines in the open disk $\Omega$.  As expected, from the recent result of
\'Alvarez-Paiva in \cite{Alvarez13}, the non-reversible Finsler
function $F$ is the sum of a reversible projective metric and an exact
$1$-form.  

\section{Examples}
\label{sec:examples}

In the previous subsection, we have seen already an example,
\eqref{ssc:ex3}, of a spray that is metrizable by a Finsler function
of scalar flag curvature. We have tested the Finsler metrizability of
this spray using the conditions of Theorem \ref{scalar_flag}. In this
section we will use again the conditions of Theorem \ref{scalar_flag} to
test whether or not some other examples of sprays are Finsler
metrizable. We will also see that the regularity condition iv) of
Theorem \ref{scalar_flag} can be relaxed and we can search for sprays
metrizable by conic pseudo-, or degenerate Finsler functions. 

\subsection{A spray metrizable by a conic pseudo-Finsler function} 

Consider the following affine spray on some domain $M\subset
\mathbb{R}^2$, where two smooth functions $\phi$ and $\psi$ are defined,
\begin{eqnarray*}
  S=y^1\frac{\partial}{\partial x^1} + y^2\frac{\partial}{\partial x^2}
  - \phi(x^1, x^2) (y^1)^2 \frac{\partial}{\partial y^1} - \psi(x^1,
  x^2) (y^2)^2 \frac{\partial}{\partial y^2}. \end{eqnarray*} 
Using formulae \eqref{localphi},  the local components of the
corresponding Jacobi endomorphism  are given by
\begin{eqnarray*}
  R^1_1= -\phi_{x^2}y^1y^2, \quad R^1_2 = \phi_{x^2}(y^1)^2, \quad  R^2_1 =
  \psi_{x^1} (y^2)^2, \quad R^2_2=
  -\psi_{x^1}y^1y^2 . \end{eqnarray*}
According to formula \eqref{ricr}, the Ricci scalar is given by 
\begin{eqnarray*} \rho= R^1_1+R^2_2=-y^1y^2(\phi_{x^2} +
  \psi_{x^1}). \end{eqnarray*}
The case when $\phi_{x^2} =-  \psi_{x^1} \neq 0$ has been studied in
Example 8.2.4 from \cite{Shen01}. In this case, we have that the Ricci
scalar is $\rho=0$ while $\Phi\neq 0$ and hence $S$ is not Finsler metrizable. 

We pay now attention to the case $\rho\neq 0$. In this case, using the four conditions of Theorem \ref{scalar_flag}, we will prove that $S$ is Finsler metrizable if and only if there exists
a constant $c\in \mathbb{R}\setminus\{0,1\}$, such that 
\begin{eqnarray}
  c\phi_{x^2}=(1-c)\psi_{x^1}. \label{exfpc}
\end{eqnarray}

Since $S$ is a spray on a $2$-dimensional manifold, it follows that it is isotropic and hence the first condition of Theorem
\ref{scalar_flag} is satisfied. The two components of the semi-basic
$1$-form $\alpha=\alpha_1dx^1+\alpha_2 dx^2$, which appear in the expression \eqref{isophi} of
the Jacobi endomorphism, are given by \cite[(4.4)]{BM13}:
\begin{eqnarray*}
  \alpha_1=\frac{R^2_2}{y^1}=-\psi_{x^1} y^2, \quad
  \alpha_2=\frac{R^1_1}{y^2}=-\phi_{x^2} y^1. \end{eqnarray*} 
The last three conditions of Theorem \ref{scalar_flag}  refer to the
semi-basic $1$-form $\alpha/\rho$, which is given by
\begin{eqnarray*}
  \frac{\alpha}{\rho} = \frac{\psi_{x^1}}{(\phi_{x^2} +
    \psi_{x^1})y^1}dx^1 + \frac{\phi_{x^2}}{(\phi_{x^2} +
    \psi_{x^1})y^2}dx^2. \label{exfpar} \end{eqnarray*}
For the second condition of Theorem \ref{scalar_flag}, one can
immediately check that $d_J(\alpha/\rho)=0$ and therefore there exists
a function $f$ defined on the conic region $A=\{(x^1,x^2, y^1, y^2)\in
TM, y^1>0, y^2>0\}$ of $T_0M$, such that $\alpha/\rho=d_Jf$. The function $f$ is given by 
\begin{eqnarray}
  f(x,y)=\frac{1}{\phi_{x^2} + \psi_{x^1}}\left( \psi_{x_1}\ln y^1 +
    \phi_{x_2}\ln y^2\right). \label{exfpf} \end{eqnarray}  
For the third condition of Theorem \ref{scalar_flag}, we have to test
if $d_hf$ is a basic $1$-form. For the spray $S$, the local
coefficients $N^i_j$, of the nonlinear connection are given by 
\begin{eqnarray*} N^1_1=\phi y^1, \quad N^1_2=N^2_1=0, \quad
  N^2_2=\psi y^2. \end{eqnarray*} It follows that 
\begin{eqnarray*} 
  d_hf  =  \frac{\delta f}{\delta x^1}dx^1 + \frac{\delta f}{\delta
    x^2}dx^2, \quad 
  \frac{\delta f}{\delta x^1}  =  \frac{\partial f}{\partial x^1} -
  \frac{\phi \psi_{x^1}}{\phi_{x^2} +
    \psi_{x^1}}, \quad \frac{\delta f}{\delta x^2}  =  \frac{\partial f}{\partial x^2} -
  \frac{\psi \phi_{x^2}}{\phi_{x^2} + \psi_{x^1}}.   \end{eqnarray*} 
Therefore $d_hf$ is a basic $1$-form if and only if there exist two
real constant $c_1$ and $c_2$ such that 
\begin{eqnarray}
  \frac{\psi_{x^1}}{\phi_{x^2} + \psi_{x^1}} = c_1, \quad  \frac{\phi_{x^2}}{\phi_{x^2} +
    \psi_{x^1}}=c_2. \label{exfpc12} \end{eqnarray} Expression \eqref{exfpf} and the
condition $\mathbb{C}(f)=1$ implies $c_1+c_2=1$. Formula
\eqref{exfpc12} is equivalent to formula \eqref{exfpc}, for $c=c_1$
and $c_2=1-c$. 

We will show that, within the given assumptions \eqref{exfpc}, the
last condition of Theorem \ref{scalar_flag} is satisfied. We have that 
\begin{eqnarray*}
  \frac{\alpha}{\rho} = \frac{c}{y^1}dx^1 + \frac{1-c}{y^2}dx^2 \end{eqnarray*}
and therefore 
\begin{eqnarray*}
  d\left(\frac{\alpha}{\rho}\right) + 2i_{\mathbb F}
  \frac{\alpha}{\rho}\wedge \frac{\alpha}{\rho} =
  \frac{c(2c-1)}{(y^1)^2} \delta y^1 \wedge dx^1 +
  \frac{c(2-2c)}{y^1y^2}(\delta y^1 \wedge dx^2 + \delta y^2 \wedge
  dx^1) + \frac{(1-c)(1-2c)}{(y^2)^2} \delta y^2\wedge dx^2,
\end{eqnarray*} which is non-degenerate and hence it is a symplectic
form on $A\subset T_0M$. 

We have shown that the spray $S$ is Finsler metrizable if and only if
the condition \eqref{exfpc} is satisfied. We will show now how we can
construct the Finsler function that metricizes the spray. To simplify
the calculations, we choose the constant $c=1/2$ and the functions
$\phi(x^1, x^2)=\psi(x^1, x^2)=2g'(x^1+x^2)/g(x^1+x^2)$, where $g(t)$ is a
non-vanishing smooth function. In this case, one can see that the condition
\eqref{exfpc} is satisfied. 

For this choice we have that the basic $1$-form $d_hf$ is given by  
\begin{eqnarray*} d_hf= - \frac{g'}{g} dx^1 - \frac{g'}{g} dx^2
  = da, \quad a(x^1, x^2) = -\ln g(x^1+x^2). \end{eqnarray*}  
According to formula \eqref{expf}, it follows that 
\begin{eqnarray*}
  F(x,y)=\exp(f-a)=\frac{\sqrt{y^1y^2}}{g(x^1+x^2)}, \end{eqnarray*}
metricizes the spray $S$ for the given choice of the functions $\phi$ and
$\psi$. The scalar flag curvature is given by formula \eqref{krf}, and
for the above Finsler function is 
\begin{eqnarray*}
  \kappa = \frac{\rho}{F^2}= 4(g''g-(g')^2). \end{eqnarray*}
For the particular case, when $g(t)=t/2$ we obtain the case of constant
sectional curvature $\kappa=-1$ studied in \cite[\S 5.4]{BM13}.

\subsection{A spray metrizable by a degenerate Finsler function}
We present now an example of a spray that is metrizable by a
degenerate Finsler function of scalar flag curvature. This means that the first three conditions
of Theorem \ref{scalar_flag}  are satisfied, while the last one it is
not. On $M=\mathbb{R}^2$, consider the following system of second
order ordinary differential equations:
\begin{eqnarray}
  \frac{d^2x^1}{dt^2} + 2\frac{dx^1}{dt}\frac{dx^2}{dt}=0, \quad
  \frac{d^2x^2}{dt^2} - \left(\frac{dx^2}{dt}\right)^2=0. \label{systemdf}
\end{eqnarray}
The corresponding spray $S\in \mathfrak{X}(TM)$ is given by 
\begin{eqnarray*}
  S=y^1\frac{\partial}{\partial x^1} + y^2\frac{\partial}{\partial x^2}
  - 2y^1y^2\frac{\partial}{\partial y^1} +
  (y^2)^2\frac{\partial}{\partial y^2}. \label{spraydf}
\end{eqnarray*}  
Using the formulae \eqref{localphi} and \eqref{ricr}, the local components of the
corresponding Jacobi endomorphism and the Ricci scalar are given by
\begin{eqnarray*}
  R^1_1= -2(y^2)^2, \quad R^2_2= 0, \quad \rho =  -2(y^2)^2.
\end{eqnarray*}
Since $S$ is a two-dimensional spray, it follows that it is isotropic
and hence first condition of Theorem \ref{scalar_flag}  is
satisfied. The semi-basic $1$-form $\alpha/\rho=\alpha_1/\rho dx^2 +
\alpha_2/\rho dx^2$ has the components:
\begin{eqnarray*}
  \frac{\alpha_1}{\rho}=\frac{R^2_2}{y^1\rho}=0, \quad
  \frac{\alpha_2}{\rho}=\frac{R^1_1}{y^2\rho}=\frac{1}{y^2}. \label{ardf}
\end{eqnarray*} 
From the above formulae, one can immediately check that
$d_J(\alpha/\rho)=0$ and hence the second condition of Theorem \ref{scalar_flag}  is
satisfied. Moreover, we have that there exists a function $f\in
C^{\infty}(T_0M)$ such that
\begin{eqnarray*}
  \frac{\alpha}{\rho}=d_Jf, \quad \textrm{for} \ f(x,y)=\ln
  |y^2|. \end{eqnarray*}
Third condition of Theorem \ref{scalar_flag}  is
satisfied if and only if $d_hf$ is a basic $1$-form. By direct
calculation we have that this is true, since 
$d_hf=dx^2$. More than that, for $a(x^1, x^2)=x^2$, we have that
$d_hf=da$.  Therefore the function  
\begin{eqnarray*}
  F(x,y)=\exp(f(x,y)-a(x))=\exp(-x^2)y^2 \end{eqnarray*} is a degenerate
Finsler function that metricizes the given system \eqref{systemdf}.
This degenerate Finsler function has scalar flag curvature, given by formula
\eqref{krf}, which in our case is
\begin{eqnarray*}
  \kappa = \frac{\rho}{F^2}= \frac{-2}{\exp(-x^2)}. \end{eqnarray*}
It can be directly checked that any solution of the system
\eqref{systemdf} is also a solution of the Euler-Lagrange equations
for $F^2$. Some other non-homogeneous Lagrangian functions that
metricizes the system \eqref{systemdf} where determined in \cite[Ex. 7.10]{AT92}. 

\subsection{A spray that is not Finsler metrizable}
We consider now an example of a spray that is not Finsler
metrizable, and this is due to the fact that the third condition
of Theorem \ref{scalar_flag}  is not satisfied. On $M=\mathbb{R}^2$,
we consider the following system of second
order ordinary differential equations:
\begin{eqnarray}
  \frac{d^2x^1}{dt^2} + \left(\frac{dx^1}{dt}\right)^2 + \left(\frac{dx^2}{dt}\right)^2=0, \quad
  \frac{d^2x^2}{dt^2} +4 \frac{dx^1}{dt}\frac{dx^2}{dt}=0. \label{systemnf}
\end{eqnarray}
The above system can be identified with a spray $S\in
\mathfrak{X}(TM)$, which is given by 
\begin{eqnarray*}
  S=y^1\frac{\partial}{\partial x^1} + y^2\frac{\partial}{\partial x^2}
  - \left((y^1)^2+(y^2)^2\right)\frac{\partial}{\partial y^1} -
  4y^1y^2\frac{\partial}{\partial y^2}. \label{spraynf}
\end{eqnarray*}  
We make use of formulae \eqref{localphi} and \eqref{ricr} to compute the local components of the
corresponding Jacobi endomorphism and the Ricci scalar, which are given by
\begin{eqnarray*}
  R^1_1= -(y^2)^2, \quad R^2_2= -2(y^1)^2, \quad \rho =  -2(y^1)^2 - (y^2)^2.
\end{eqnarray*}
Again, the spray $S$ is two-dimensional and hence it is isotropic,
which means that the first condition of Theorem \ref{scalar_flag}  is
satisfied. The other conditions refer to the semi-basic $1$-form $\alpha/\rho=(\alpha_1/\rho) dx^2 +
(\alpha_2/\rho) dx^2$, whose components are given by:
\begin{eqnarray*}
  \frac{\alpha_1}{\rho}=\frac{R^2_2}{y^1\rho}=\frac{2y^1}{2(y^1)^2 + (y^2)^2}, \quad
  \frac{\alpha_2}{\rho}=\frac{R^1_1}{y^2\rho}=\frac{y^2}{2(y^1)^2 + (y^2)^2}. \label{arnf}
\end{eqnarray*} 
From the above formulae, it follows that $d_J(\alpha/\rho)=0$, which
means that the second condition of Theorem \ref{scalar_flag}  is
satisfied. Therefore, there exists a function $f\in
C^{\infty}(T_0M)$ such that
\begin{eqnarray*}
  \frac{\alpha}{\rho}=d_Jf, \quad \textrm{for} \ f(x,y)=\ln(2(y^1)^2 + (y^2)^2). \end{eqnarray*}
For the above considered function $f$ we can check that $d_hf$ is not
a basic $1$-form. It follows then that third condition  of Theorem
\ref{scalar_flag} is not satisfied and consequently the spray is not
Finsler metrizable. The system \eqref{systemnf} has been considered
in \cite[Ex. 7.2]{AT92}, where it has been shown that it is not
metrizable, using different techniques. 

 \subsection*{Acknowledgments} The work of I.B. has been supported by the
 Romanian National Authority for Scientific Research,
 CNCS UEFISCDI, project number
 PN-II-ID-PCE-2012-4-0131. The work of
 Z.M. has been supported by the Hungarian Scientific Research Fund (OTKA) Grant K67617.

\end{document}